\documentclass[12pt,a4paper]{article}
\usepackage{amsmath,amsfonts}
\usepackage{amsthm}
\usepackage{upref}
\usepackage{enumerate}
\usepackage[margin=3cm]{geometry}
\usepackage{authblk}
\usepackage{tikz}
\usepackage[colorlinks=true]{hyperref}

\date{September 10, 2015}

\title{Periodic-parabolic eigenvalue problems with a large parameter
  and degeneration}
\author{Daniel Daners}%
\author{Christopher Thornett\thanks{Current address: Courant Institute
    of Mathematical Sciences, New York University, 251 Mercer Street,
    New York, NY~10012, USA}}%
\affil{School of Mathematics and Statistics University of Syndey, NSW
  2006, Australia\authorcr \nolinkurl{daniel.daners@sydney.edu.au}%
  \qquad\nolinkurl{thornett@cims.nyu.edu}}%

\makeatletter
\hypersetup{pdftitle={\@title}}%
\hypersetup{pdfauthor={Daniel Daners, Christopher Thornett}}%
\makeatother

\numberwithin{equation}{section}
\numberwithin{figure}{section}

\theoremstyle{plain}
\newtheorem{theorem}{Theorem}[section]
\newtheorem{proposition}[theorem]{Proposition}
\newtheorem{lemma}[theorem]{Lemma}
\newtheorem{corollary}[theorem]{Corollary}

\theoremstyle{definition}
\newtheorem{definition}[theorem]{Definition}
\newtheorem{assumption}[theorem]{Assumption}

\theoremstyle{remark}

\newtheorem{remark}[theorem]{Remark}

\newtheorem{example}[theorem]{Example}

\DeclareMathOperator{\divergence}{div}

\DeclareMathOperator{\spr}{spr}
\DeclareMathOperator{\interior}{int}
\DeclareMathOperator{\supp}{supp}
\DeclareMathOperator{\one}{\mathbf 1}

\let\oldthebibliography\thebibliography
\renewcommand\thebibliography[1]{
  \oldthebibliography{#1}
  \setlength{\parskip}{0pt}
  \setlength{\itemsep}{0pt plus 0.3ex}
  \small
}

\begin{document}

\maketitle

\begin{abstract}
  We consider a periodic-parabolic eigenvalue problem with a
  non-negative potential $\lambda m$ vanishing on a non-cylindrical
  domain $D_m$ satisfying conditions similar to those for the parabolic
  maximum principle. We show that the limit as $\lambda\to\infty$ leads
  a periodic-parabolic problem on $D_m$ having a unique
  periodic-parabolic principal eigenvalue and eigenfunction. We
  substantially improve a result from [Du \& Peng,
  Trans. Amer. Math. Soc. 364 (2012), p.~6039--6070]. At the same time
  we offer a different approach based on a periodic-parabolic initial
  boundary value problem. The results are motivated by an analysis of
  the asymptotic behaviour of positive solutions to semilinear logistic
  periodic-parabolic problems with temporal and spacial degeneracies.
\end{abstract}

{ %Provide MSC and Keywords
\renewcommand{\thefootnote}{}
}
\section{Introduction}
\label{sec:intro}
We consider a periodic-parabolic eigenvalue problem arising in the study
of the asymptotic behaviour of positive solutions to a $T$-periodic
logistic type population problem such as first studied in
\cite{hess:91:ppb,beltramo:84:pep} and later in
\cite{alvarez:08:abp,alvarez:08:cabp,du:99:bsc,du:12:ple,du:09:ltl,lopez:96:mpe}. The
limiting behaviour of the eigenvalue problem allows to deduce
information about the corresponding logistic-type semilinear
problem. Our focus is in on the case of temporal and spacial
degeneracies motivated in particular in \cite{du:12:ple}.

More precisely, we are interested in the behaviour of the principal
eigenvalue for the periodic-parabolic eigenvalue problem
\begin{equation}
  \label{eq:pevp}
  \begin{aligned}
    \frac{\partial u}{\partial t}+\mathcal A(t)u+\lambda m(x,t)u
    &=\mu(\lambda)u
    &&\text{in $\Omega\times(0,T)$,}\\
    \mathcal B(t)u&=0
    &&\text{in $\partial\Omega\times(0,T)$,}\\
    u(x,0)&=u(x,T)&&\text{in $\Omega$,}
  \end{aligned}
\end{equation}
as $\lambda\to\infty$, where $m\in
L^\infty\bigl(\Omega\times(0,T)\bigr)$ is a non-negative weight function
that has a non-trivial zero set satisfying suitable
assumptions. Moreover, $\Omega\subseteq\mathbb R^N$ is a bounded domain,
and
\begin{equation}
  \label{eq:A}
  \mathcal A(t)u:=
  -\divergence\bigl(D(x,t)\nabla u+a(x,t)u\bigr)+\bigl(b(x,t)\cdot\nabla
  u+c_0(x,t)u\bigr)
\end{equation}
is a uniformly strongly elliptic operator with bounded and measurable
coefficients and $\mathcal B(t)$ a boundary operator of Dirichlet,
Neumann or Robin type (for precise assumptions see
Section~\ref{sec:weak-solutions}).

As in \cite{hess:91:ppb}, a principal eigenvalue of \eqref{eq:pevp} is
an eigenvalue having a \emph{positive} eigenvector. If $m(x,t)>0$ on
$\Omega\times(0,T)$ nothing interesting happens, so we focus on the case
where $m(x,t)=0$ in some region $D_m\subseteq\Omega\times[0,T]$ of
non-zero measure. Such problems have been looked at in particular for
the corresponding elliptic problem in
\cite{alvarez:08:abp,dancer:11:lps,lopez:96:mpe}. The most general
weights $m$ are considered in
\cite{alvarez:14:qac,du:12:ple,peng:15:lbc}, where spacial and temporal
degeneration is allowed. Our aim is to simplify and generalise some of
these results using an alternative method and allowing fully
non-automonous operators $(\mathcal A(t),\mathcal B(t))$ including the
principal part.

The approach we take is quite different from previous work and related
to the one used in \cite{daners:13:epc} for elliptic systems. Rather
than studying the eigenvalue problem \eqref{eq:pevp} directly we study
what happens to the solution to
\begin{equation}
  \label{eq:hivp}
  \begin{aligned}
    \frac{\partial u}{\partial t}+\mathcal A(t)u+\lambda m(x,t)u
    &=0
    &&\text{in $\Omega\times(s,T)$,}\\
    \mathcal B(t)u&=0
    &&\text{in $\partial\Omega\times(s,T)$,}\\
    u(x,s)&=u_s(x)&&\text{in $\Omega$,}
  \end{aligned}
\end{equation}
as $\lambda\to\infty$, where $s\in[0,T)$. We consider the behaviour of
weak solutions of \eqref{eq:hivp} with a non-zero right hand side as
$\lambda\to\infty$ in Section~\ref{sec:weak-solutions}. In
Section~\ref{sec:Lp-solutions} we show that for every initial value
$u_0\in L^p(\Omega)$ the problem \eqref{eq:hivp} has a unique solution
$u\in C([s,T],L^p(\Omega))$. This in particular allows us to define the
evolution operator $U_\lambda(t,s)$ by
\begin{equation}
  \label{eq:evolution-operator}
  U_\lambda(t,s)u_s:=u(t).
\end{equation}
Letting $\lambda\to\infty$ we get an evolution operator $U_\infty(t,s)$
(not necessarily strongly continuous at $t=s$). We show that
$U_\lambda(t,s)$ has uniform Gaussian kernel estimates, which lead to
uniform $L^p$-$L^q$ estimates for solutions of \eqref{eq:hivp} and the
limit problem as $\lambda\to\infty$.

Our main results on the convergence and existence of principal
periodic-parabolic eigenvalues and eigenfunctions are then given in
Section~\ref{sec:periodic}, in particular
Theorem~\ref{thm:pp-eigenvalue}. The idea is to use continuity
properties of eigenvalues similarly as in
\cite{daners:00:epp,daners:05:pse}. The result generalises
\cite[Theorem~3.3 and~3.4]{du:12:ple} significantly by allowing much
more general conditions on $m$. In the final section we show that our
conditions on $m$ are in some sense optimal.

\section{Convergence of weak solutions}
\label{sec:weak-solutions}
Before stating our main result we make our assumptions precise. We
consider a boundary operator of the form \eqref{eq:A} with $D\in
L^\infty(\Omega\times(0,T),\mathbb R^{N\times N})$, $a,b\in
L^\infty(\Omega\times(0,T),\mathbb R^N)$ and $c_0\in
L^\infty(\Omega\times(0,T),\mathbb R)$. We assume that $\mathcal A(t)$
is uniformly strongly elliptic, that is, the matrix $D(x,t)$ is positive
definite uniformly with respect to $(x,t)$. Hence there exists
$\alpha>0$ such that
\begin{displaymath}
  y^TD(x,t)y\ge\alpha|y|^2
\end{displaymath}
for all $y\in\mathbb R^N$ and almost all
$(x,t)\in\Omega\times(0,T)$. We admit boundary operators of the form
\begin{displaymath}
  \mathcal B(t)u:=
  \begin{cases}
    u|_{\partial\Omega}&\text{Dirichlet boundary operator}\\
    \bigl(D\nabla u+au\bigr)\cdot\nu+b_0u&\text{Neumann or Robin
      boundary operator}
  \end{cases}
\end{displaymath}
where $\nu$ is the outward pointing unit normal to $\Omega$ and $b_0\in
L^\infty(\partial\Omega)$. If $b_0=0$ we have (natural) Neumann boundary
conditions, and if $b_0\neq 0$ we have Robin boundary conditions. In
case of Dirichlet boundary conditions we admit any bounded domain
$\Omega\subseteq\mathbb R^N$. In case of Neumann or Robin boundary we
assume that $\Omega$ is a Lipschitz domain. In we can then assume
without loss of generality that $b_0>0$ on $\partial\Omega$ as shown in
\cite{daners:09:ipg}. We could also have Dirichlet and Robin boundary
conditions on disjoint parts of $\partial\Omega$. We finally assume that
$m\in L^\infty(\Omega\times(0,T))$ is non-negative and that
$\lambda\in\mathbb R$.

We use the theory of variational evolution equations as presented in
\cite{dautray:92:man5,lions:72:nbv} to study the initial value problem
\begin{equation}
  \label{eq:ivbp}
  \begin{aligned}
    \frac{\partial u}{\partial t}+\mathcal A(t)u+\lambda mu
    &=f(x,t)
    &&\text{in $\Omega\times(s,T)$,}\\
    \mathcal B(t)u&=0
    &&\text{in $\partial\Omega\times(s,T)$,}\\
    u(x,s)&=u_s&&\text{in $\Omega$,}
  \end{aligned}
\end{equation}
as $\lambda\to\infty$. We first look at the $L^2$-theory, and then by
means of heat kernel estimates generalise to an $L^p$-theory.  For
$t\in[0,T]$ we introduce the bilinear forms
\begin{displaymath}
  a(t,u,v):=\int_\Omega\bigl(D\nabla u+au\bigr)\nabla v
  +\bigl(b\cdot\nabla u+c_0u\bigr)v\,dx
  +\int_{\partial\Omega}b_0uv\,d\sigma,
\end{displaymath}
where $d\sigma$ denotes integration with respect to surface measure on
$\partial\Omega$.  That bilinear form is defined on the space
\begin{displaymath}
  V:=
  \begin{cases}
    H_0^1(\Omega)&\text{in case of Dirichlet boundary conditions,}\\
    H^1(\Omega)&\text{in case of Robin or Neumann boundary conditions.}\\
  \end{cases}
\end{displaymath}
For Dirichlet (or Neumann) boundary conditions the boundary integral is
not present. From the assumptions on the coefficients of $\mathcal A(t)$
there exists a constant $M\geq 0$ such that
\begin{displaymath}
  |a(t,u,v)|\leq M\|u\|_V\|v\|_V
\end{displaymath}
for all $u,v\in V$ and $t\in[0,T]$. This is also true for Robin boundary
conditions since in that case we assume that $\Omega$ is a Lipschitz
domain and therefore we can use a trace inequality to estimate the
boundary integral. Further, one can show that
\begin{displaymath}
  \frac{\alpha}{2}\|u\|_V^2\leq a(t,u,u)+\gamma\|u\|_2^2
\end{displaymath}
for all $u\in V$ and $t\in[0,T]$, where
\begin{equation}
  \label{eq:gamma}
  \gamma\geq\gamma_0
  :=\frac{\|a\|_\infty+\|b\|_\infty}{2\alpha}+\|c_0^-\|_\infty
\end{equation}
(see for instance \cite[Prop~2.4]{daners:09:ipg}). Naturally,
$V\hookrightarrow L^2(\Omega)$ is compactly embedded. Identifying
$L^2(\Omega)$ with its dual,
\begin{math}
  V\hookrightarrow L_2(\Omega)\hookrightarrow V'
\end{math}
are dense and compact embeddings, where $V'$ is the dual of $V$. In that
case, duality is given by
\begin{displaymath}
  \langle u,v\rangle:=\int_\Omega u(x)v(x)\,dx
\end{displaymath}
for $u\in L^2(\Omega),\,v\in V$. Given $f\in L^2((s,T),V')$ we call
$u\in L^2((s,T),V)$ a weak solution of \eqref{eq:ivbp} if
\begin{multline}
  \label{eq:weaksol}
  -\int_s^T\langle \dot v(t),u(t)\rangle\,dt
  -\langle u_s,v(s)\rangle+\int_s^Ta(t,u(t),v(t))\,dt\\
  +\lambda\int_s^T\langle mu(t),v(t)\rangle\,dt
  =\int_s^T\langle f(t),v(t)\rangle\,dt
\end{multline}
for all $v\in C^\infty_c([s,T))\otimes V$. We then introduce the
spaces
\begin{displaymath}
  W(s,T,V,V')
  :=\bigl\{u\in L^2((s,T),V)\colon \dot u\in L^2((s,T),V')\bigr\},
\end{displaymath}
for $s\in[0,T)$, where $V'$ is the dual space of $V$ and $\dot u$ is the
derivative with respect to $t$ in the sense of distributions with values
in $V'$. The space $W(s,T,V,V')$ is a Hilbert space with the norm
\begin{displaymath}
  \|u\|_W
  :=\Bigl(\int_s^T\|u(t)\|_V^2\,dt
  +\int_s^T\|\dot u(t)\|_{V'}\,dt\Bigr)^{1/2}.
\end{displaymath}
The space $W(s,T,V,V')$ has some useful properties. First of all we have the
embedding
\begin{equation}
  \label{eq:Wcontinuous}
  W(s,T,V,V')\hookrightarrow C([s,T], L^2(\Omega)).
\end{equation}
For this reason it makes sense to write $u(t)$ for $t\in[0,T]$. Moreover
the embedding
\begin{equation}
  \label{eq:Wcompact}
  W(s,T,V,V')\hookrightarrow L^2((s,T), L^2(\Omega))
\end{equation}
is compact.  We also have the formula of integration by parts
\begin{equation}
  \label{eq:ip}
  \int_s^t\langle \dot u(\tau),v(\tau)\rangle\,d\tau
  +\int_s^t\langle u(\tau),\dot v(\tau)\rangle\,d\tau
  =\langle u(t),\dot v(t)\rangle-\langle u(s),\dot v(s)\rangle
\end{equation}
for all $u,v\in W(0,T,V,V')$ and $0\leq s\leq t\leq T$. Finally, one may
show that $C_c^\infty([s,T))\otimes V$ is dense in $\{v\in
W(s,T,V,V')\colon v(T)=0\}$. This implies that we may test all $v\in
W(s,T,V,V')$ with $v(T)=0$ in the definition \eqref{eq:weaksol} of a
weak solution. By defining the operators $A(t)\in\mathcal L(V,V')$ by
$\langle A(t)u,v\rangle=a(t,u,v)$, we note that $u$ is a weak solution
if and only if $u\in W(s,T,V,V')$ and
\begin{equation}
  \label{eq:weaksolalt}
  \begin{aligned}
    \dot u(t) + A(t)u(t) + \lambda m(t)u (t)&= f(t)
    \qquad\text{for $t\in(s,T]$,}\\
    u(s)&=u_s&&
  \end{aligned}
\end{equation}
where equality in the first line is in the sense of
$L^2\bigl((s,T),V'\bigr)$. For all these facts see
\cite[Section~XVIII.\S1.2]{dautray:92:man5} and
\cite[Theorem~1.5.1]{lions:69:qmr} for the compact embedding
\eqref{eq:Wcompact}. We first prove some a priori estimates.
\begin{proposition}
  \label{prop:apriori}
  Suppose that $f\in L^2\bigl((0,T),V'\bigr)$ and $\gamma$ as in
  \eqref{eq:gamma}. If $u$ is a weak solution of \eqref{eq:ivbp}, then
  \begin{multline*}
    \frac{1}{2}\|u(t)\|_2^2
    +\frac{\alpha}{4}\int_s^t\|e^{\gamma(t-\tau)}u(\tau)\|_V^2\,d\tau
    +\lambda\int_s^t\langle m(\tau)u(\tau),u(\tau)\rangle
    e^{2\gamma(t-\tau)}\,d\tau\\
    \leq\frac{1}{2}\|u(s)e^{\gamma(t-s)}\|_2^2
    +\frac{1}{\alpha}
    \int_s^t e^{2\gamma(t-\tau)}\|f(\tau)\|_{V'}^2\,d\tau
  \end{multline*}
  for all $0\leq s\leq t\leq T$ and all $\lambda\geq 0$.
\end{proposition}
\begin{proof}
  As $e^{-\gamma t}u(t)$ solves \eqref{eq:ivbp} with $\mathcal A$
  replaced by $\mathcal A+\gamma$ and $f(t)$ replaced by $e^{-\gamma
    t}f(t)$ we conclude from the definition of a weak solution and
  \eqref{eq:ip} that
  \begin{displaymath}
    \begin{split}
      \frac{\alpha}{2}\int_s^t\|e^{-\gamma\tau} &u(\tau)\|_V^2\,d\tau
      +\lambda\int_s^t\langle m(\tau)u(\tau),u(\tau)\rangle
      e^{-2\gamma\tau}\,d\tau\\
      &\leq \int_s^t
      a(\tau,e^{-\gamma\tau}u(\tau),e^{-\gamma\tau}u(\tau))\,d\tau
      +\lambda\int_s^t\langle e^{-\gamma\tau}m(\tau)u(\tau),
      e^{-\gamma\tau}u(\tau)\rangle \,d\tau\\
      &=-\frac{1}{2}\|u(t)e^{\gamma t}\|_2^2
      +\frac{1}{2}\|u(s)e^{-\gamma s}\|_2^2
      +\int_s^t\langle e^{-\gamma\tau}f(\tau),e^{-\gamma\tau}u(\tau)\rangle\,d\tau.
    \end{split}
  \end{displaymath}
  By an elementary inequality
  \begin{displaymath}
    \int_s^t\langle e^{-\gamma\tau}f(\tau),e^{-\gamma\tau}u(\tau)\rangle\,d\tau
    \leq\frac{\alpha}{4}\int_s^t\|e^{-\gamma\tau}u(\tau)\|_V^2\,d\tau
    +\frac{1}{\alpha}\int_s^t e^{-2\gamma\tau}\|f(\tau)\|_{V'}^2\,d\tau.
  \end{displaymath}
  Putting everything together and multiplying the inequality by
  $e^{2\gamma t}$ we get the required estimate.
\end{proof}
Using the above estimate we can get a compactness result.
\begin{theorem}
  \label{thm:compact}
  Suppose that $(f_n)$ is a bounded sequence in $L_2((0,T),V')$ and that
  $\lambda_n\to\infty$.  Also assume $u_{0n}$ is bounded in
  $L^2(\Omega)$. Let $u_n$ be the solution of \eqref{eq:ivbp} with
  $\lambda$ replaced by $\lambda_n$, $f$ replaced by $f_n$ and $u_0$
  replaced by $u_{0n}$.  For $\varepsilon>0$ let
  \begin{displaymath}
    S_\varepsilon
    :=\{(x,t)\in\Omega\times(0,T)\colon m(x,t)\geq\varepsilon\}
    \subseteq\Omega\times(0,T)
  \end{displaymath}
  Then the following assertions are true.
  \begin{enumerate}[\upshape (i)]
  \item $(u_n)$ is bounded in $L^\infty\bigl((0,T),L^2(\Omega)\bigr)$
    and in $L^2\bigl((0,T),V\bigr)$.
  \item $u_n\to 0$ in $L^2(S_\varepsilon)$ for all $\varepsilon>0$.
  \item There exists a subsequence $(n_k)$ such that
    $u_{n_k}\rightharpoonup u$ weakly in $L^2\bigl((0,T),V\bigr)$,
    $u_{0{n_k}}\rightharpoonup u_0$ weakly in $L^2(\Omega)$ and
    $f_{n_k}\rightharpoonup f$ weakly in
    $L^2\bigl((0,T),V'\bigr)$. Moreover,
    \begin{multline}
      \label{eq:weaksol-limit}
      -\int_0^T\langle\dot v(t),u(t)\rangle\,dt-\langle u_0,v(0)\rangle
      +\int_0^Ta(t,u(t),v(t))\,dt\\
      =\int_0^T\langle f(t),v(t)\rangle\,dt
    \end{multline}
    for all $v\in W(0,T,V,V')$ with $v=0$ on
    \begin{equation}
      \label{eq:S_0}
      S_0:=\bigl\{(x,t)\in\Omega\times(0,T)\colon m(x,t)>0\bigr\}
      =\bigcap_{\varepsilon>0}S_\varepsilon.
    \end{equation}
  \end{enumerate}
\end{theorem}
\begin{proof}
  (i) By Proposition ~\ref{prop:apriori},
  \begin{equation}
    \label{eq:u_n-bounded}
    \|u_n(t)\|_2^2
    +\frac{\alpha}{4}\int_0^te^{\gamma(t-\tau)}\|u_n(\tau)\|_V^2\,d\tau
    \leq e^{2\gamma T}\|u_{0n}\|_2^2+\frac{2e^{2\gamma T}}{\alpha}
    \|f_n\|_{L^2((0,T),V')}
  \end{equation}
  for all $0\le t\le T$, which remains uniformly bounded as with respect
  to $t\in [0,T]$ as $n\to\infty$. Hence, $(u_n)$ is bounded in
  $L^\infty\bigl((0,T),L^2(\Omega)\bigr)$. Since $\gamma\geq 0$ we have
  $e^{\gamma(T-\tau)}\geq 1$ for all $\tau\in[0,T]$. Setting $t=T$ it
  therefore follows from \eqref{eq:u_n-bounded} that $(u_n)$ is bounded
  in $L^2\bigl((0,T),V\bigr)$.

  (ii) By Proposition~\ref{prop:apriori} the sequence
  \begin{displaymath}
    \lambda_n
    \int_0^T\langle m(\tau)u_n(\tau),u_n(\tau)\rangle\,d\tau
  \end{displaymath}
  remains bounded as $n\to\infty$. As $\lambda_n\to\infty$ we
  conclude that
  \begin{displaymath}
    \|u_{n_k}\|_{L_2(U_\varepsilon)}
    \leq\frac{1}{\varepsilon}
    \int_0^T\langle m(\tau)u_{n_k}(\tau),u_{n_k}(\tau)\rangle\,d\tau
    \to 0
  \end{displaymath}
  as $k\to\infty$.

  (iii) By (i) the sequence $(u_n)$ is bounded in
  $L^2\bigl((0,T),V\bigr)$. By assumption, $u_{0n}$ is bounded in
  $L^2(\Omega)$ and $(f_n)$ is bounded in
  $L^2\bigl((0,T),V'\bigr)$. Since all spaces are Hilbert spaces, we can
  select a subsequence such that $u_{n_k}\rightharpoonup u$ weakly in
  $L^2\bigl((0,T),V\bigr)$, $u_{0{n_k}}\rightharpoonup u_0$ weakly in
  $L^2(\Omega)$ and $f_{n_k}\rightharpoonup f$ weakly in
  $L^2\bigl((0,T),V'\bigr)$. If $v\in W(0,T,V,V')$ with $v=0$ on $S_0$,
  then 
  \begin{displaymath}
    \int_0^T\langle mu_{n_k}(\tau),v(\tau)\rangle\,d\tau=0
  \end{displaymath}
  for all $k\in\mathbb N$ and thus \eqref{eq:weaksol} reduces to
  \begin{displaymath}
    -\int_0^T\langle \dot v(t),u_{n_k}(t)\rangle\,dt
    -\langle u_{0n_k},v(s)\rangle+\int_s^Ta(t,u_{n_k}(t),v(t))\,dt
    =\int_s^T\langle f_{n_k}(t),v(t)\rangle\,dt
  \end{displaymath}
  for all $k\in\mathbb N$. Now \eqref{eq:weaksol-limit} follows by
  letting $k\to\infty$.
\end{proof}
In the above theorem we make only minimal assumptions on the weight
functions $m$. In particular, we do not require that the set $S_0$ given
by \eqref{eq:S_0} is an open set, nor that it has any regularity. We can
say something more about the limit problem if we make some stronger
assumptions. We assume that $\supp(m)$ is topologically regular in the
sense that
\begin{equation}
  \label{eq:m-top-reg}
  \supp(m)=\overline{\interior(\supp(m))}.
\end{equation}
We furthermore define the possibly non-cylindrical set
\begin{equation}
  \label{eq:D_m}
  D_m:=(\Omega\times[0,T])\setminus\supp(m)
\end{equation}
and for $0\leq t\leq T$
\begin{equation}
  \label{eq:Omega_t}
  \Omega_t:=\{x\in\Omega\colon(x,t)\in D_m\}.
\end{equation}
Intuitively, the limit problem satisfies Dirichlet boundary conditions
on $\partial D_m\cap (\Omega\times(0,T])$ because the solution of the
limit problem is forced to be zero outside $D_m$.  For this to be really
true we need to assume that $\Omega_t$ is stable in the sense of Keldysh
\cite{keldysh:41:ssd} for all $t\in[0,T]$. This means that
\begin{equation}
  \label{eq:stability}
  H_0^1(\Omega_t)
  =\bigl\{w\in H_0^1(\Omega)\colon\text{$w=0$ a.e. on
    $\Omega\setminus\Omega_t$}\};
\end{equation}
see the discussions in
\cite{arendt:08:vds,cano:01:cdp,hedberg:93:ahf,lopez:96:mpe,stummel:74:pts}.
We have the following corollary, where as usual all boundary conditions
are satisfied in the weak sense.
\begin{corollary}
  \label{cor:limit-problem}
  Suppose that the assumptions of Theorem~\ref{thm:compact} are
  satisfied, that $\supp(m)$ satisfies \eqref{eq:m-top-reg} and that
  $D_m\neq\emptyset$. Let $u$ be the limit of $(u_{n_k})$ as in
  Theorem~\ref{thm:compact}{\upshape (iii)}.

  {\upshape (i)} Then $u$ is a (local) weak solution of the parabolic
  limit problem
    \begin{equation}
    \label{eq:ivbp-limit}
      \begin{aligned}
      \frac{\partial u}{\partial t}+\mathcal A(t)u
      &=f(x,t)
      &&\text{in $D_m$,}\\
      \mathcal B(t)u&=0
      &&\text{on $\partial D_m\cap (\partial\Omega\times(0,T))$}\\
      u(x,0)&=u_0(x)&&\text{on $\Omega_0$.}
      \end{aligned}
    \end{equation}

    {\upshape (ii)} Suppose that $\Omega_t$ is stable in the sense of
    \eqref{eq:stability} for all $t\in [0,T]$.  Then the solution $u$ of
    \eqref{eq:ivbp-limit} satisfies Dirichlet boundary conditions $u=0$
    on $\partial\Omega_t\cap\Omega$ for almost all $t\in(0,T]$.
\end{corollary}
\begin{proof}
  (i) Note that \eqref{eq:weaksol-limit} is equivalent to the statement
  that $u$ is a weak solution of \eqref{eq:ivbp-limit}.

  (iii) It follows from Theorem~\ref{thm:compact}(ii) and the regularity
  assumption \eqref{eq:m-top-reg} that $u=0$ on $\supp(m)$. As $u\in
  L^2\bigl((0,T),V\bigr)$ we have that $u(t)\in V\subseteq H^1(\Omega)$
  for almost all $t\in (0,T)$ with $u(t)=0$ on
  $\Omega\setminus\Omega_t$. Since stability is a local property of the
  boundary of $\Omega_t$ it follows that $u(t)\in H_0^1(\Omega_t\cap U)$
  for every open set $U$ with $U\cap\partial\Omega=\emptyset$. Hence $u$
  satisfies Dirichlet boundary conditions in the weak sense on
  $\partial\Omega_t\cap\Omega$.
\end{proof}

\section[The evolution operator and Lp-theory]%
{The evolution operator and $L^p$-theory}
\label{sec:Lp-solutions}
It is shown in \cite{daners:00:hke} that the evolution operator
$U_\lambda(t,s)$, $0\leq s\leq t\leq T$ defined in
\eqref{eq:evolution-operator} is acting on $L^p(\Omega)$ for $1\leq
p\leq\infty$. A key fact we establish is that $U_\lambda(t,s)$ has a
kernel satisfying Gaussian estmates uniformly with respect to
$\lambda\geq 0$.
\begin{theorem}
  \label{thm:unif-est}
  Let $m\in L^\infty(\Omega\times (0,\infty))$ be non-negative.  For all
  $1\leq p\leq q\leq\infty$ the evolution operator
  $U_\lambda(t,s)\in\mathcal L(L^p(\Omega),L^q(\Omega))$ is a compact
  positive irreducible operator having a kernel $k_\lambda(x,y,t,s)$
  satisfying a Gaussian estimate. More precisely, there exist constants
  $M\geq 1$, $\omega\in\mathbb R$ and $c>0$ such that
  \begin{equation}
    \label{eq:uhkest}
    0<k_{\lambda_2}(x,y,t,s)\leq k_{\lambda_1}(x,y,t,s)
    \leq Me^{\omega(t-s)}(t-s)^{-N/2}e^{-c\frac{|x-y|^2}{t-s}}
  \end{equation}
  for all $0\leq\lambda_1\leq\lambda_2<\infty$, $0\leq s<t\leq T$ and
  $x,y\in\Omega$. Moreover,
  \begin{equation}
    \label{eq:hkest}
    \|U_\lambda(t,s)\|_{\mathcal L(L^p,L^q)}
    \leq Mt^{-\frac{N}{2}(\frac{1}{p}-\frac{1}{q})}e^{\omega(t-s)}
  \end{equation}  
  If $u_0\geq 0$, then $U_\lambda(t,s)u_0$ is decreasing as
  $\lambda\to\infty$.
\end{theorem}
\begin{proof}
  From \cite[Section~6 and~8]{daners:00:hke} the evolution operator
  $U_\lambda(t,s)$ is a positive operator on $L^p(\Omega)$ with kernel
  $k_\lambda(x,y,t,s)$ satisfying a Gaussian estimate. Assume now that
  $\lambda_1\leq\lambda_2$ and that $u_0\in L^p(\Omega)$ is
  non-negative.  For $i=1,2$ set $u_i:=U_{\lambda_i}(\cdot\,,s)u_0$.  We
  want to show that $u_2\leq u_1$. Clearly $u_2$ is the solution of
  \begin{displaymath}
    \begin{aligned}
      \frac{\partial u_2}{\partial t}+\mathcal A(t)u_2+\lambda_1 mu_2
      &=-(\lambda_2-\lambda_1)mu_2
      &&\text{in $\Omega\times(s,T)$,}\\
      \mathcal B(t)u&=0
      &&\text{in $\partial\Omega\times(s,T)$,}\\
      u(x,s)&=u_0&&\text{in $\Omega$,}
    \end{aligned}
  \end{displaymath}
  and therefore by the variation of constants formula for variational
  evolution equations (see for instance \cite[Section~4]{daners:96:dpl})
  \begin{displaymath}
    0\leq u_2(t)
    =u_1(t)-(\lambda_2-\lambda_1)
    \int_s^tU_{\lambda_1}(t,\tau)m(\tau)u_2(\tau)\,d\tau.
  \end{displaymath}
  for all $s\in[0,T)$ and all $t\in[s,T]$.  Here we used that
  $u_1(t)=U_{\lambda_1}(t,s)u_0$. We already know that $u_1,u_2\geq 0$.
  As $m\geq 0$ and $\lambda_2-\lambda_1\geq 0$ and $U(t,\tau)$ is a
  positive operator it follows that $u_2\leq u_1$. In particular,
  $U_\lambda(t,s)u_0$ is decreasing in $\lambda$ if $u_0\geq 0$. In
  terms of the heat kernels the above writes
  \begin{displaymath}
    \int_\Omega k_{\lambda_2}(x,y,t,s)u_0(y)\,dy=u_2(x,t)
    \leq u_1(x,t)=\int_\Omega k_{\lambda_1}(x,y,t,s)u_0(y)\,dy
  \end{displaymath}
  for all non-negative $u_0\in L^p(\Omega)$. Hence \eqref{eq:uhkest}
  follows from \cite[Theorem~7.1]{daners:00:hke} by taking the estimate
  of the kernel for $\lambda=0$. As $|U_\lambda(t,s)u_0| \leq
  U_\lambda(t,s)|u_0|$ we get
  \begin{displaymath}
    \|U_\lambda(t,s)\|_{\mathcal L(L^p,L^q)}
    \leq\|U_0(t,s)\|_{\mathcal L(L^p,L^q)}
  \end{displaymath}
  for all $\lambda\geq 0$. Now \eqref{eq:hkest}
  follows from \cite[Corollary~7.2]{daners:00:hke}.
\end{proof}
We next look at convergence and compactness properties of the evolution
operator.
\begin{theorem}
  \label{thm:U-limit}
  Under the assumptions of Theorem~\ref{thm:unif-est}, for every $0\leq
  s<t\leq T$ and $1<p\leq q<\infty$
  \begin{equation}
    \label{eq:U-limit}
    U_\infty(t,s)
    :=\lim_{\lambda\to\infty}U_\lambda(t,s)
  \end{equation}
  exist in $\mathcal L\bigl(L^p(\Omega),L^q(\Omega)\bigr)$. Moreover,
  $U_\infty(t,s)$ is a positive compact operator on $L^p(\Omega)$ with
  kernel $k_\infty(x,y,t,s)$ satisfying a Gaussian estimate, and
  \begin{equation}
    \label{eq:U-limit-property}
    U_\infty(t,s)=U_\infty(t,\tau)U_\infty(\tau,s)
  \end{equation}
  whenever $0\leq s<\tau<t\leq T$. Finally, the linear operator defined
  by
  \begin{displaymath}
    u_0\mapsto U_\lambda(\cdot\,,s)u_0
  \end{displaymath}
  converges to the corresponding operator in $\mathcal
  L\bigl(L^p(\Omega),L^r((s,T),L^q(\Omega))\bigr)$ as
  $\lambda\to\infty$ whenever $1<r<\infty$ and  $1<p\leq q<\infty$ are
  such that
  \begin{equation}
    \label{eq:pqr-cond}
    \frac{N}{2}\Bigl(\frac{1}{p}-\frac{1}{q}\Bigr)<\frac{1}{r}.
  \end{equation}
\end{theorem}
\begin{proof}
  First we look at strong convergence of $U_\lambda(t,s)$. For $u_0\in
  L^p(\Omega)$ non-negative we know from Theorem~\ref{thm:unif-est} that
  $U_\lambda(\cdot\,,s)u_0$ decreases as $\lambda\to\infty$ on the
  cylinder $\Omega\times(s,T]$ and therefore
  \begin{displaymath}
    u(x,t):=\lim_{\lambda\to\infty}[U_\lambda(t,s)u_0](x)=
    \lim_{\lambda\to\infty}\int_\Omega k_\lambda(t,s,x,y)u_0(y)\,dy
  \end{displaymath}
  exists for all $(x,t)\in\Omega\times(s,T]$. By \eqref{eq:uhkest} also
    \begin{equation}
      \label{eq:kplim}
      0\leq
      k_\infty(t,s,x,y)
      :=\lim_{\lambda\to\infty}k_\lambda(t,s,x,y)
      \leq Me^{\omega(t-s)}(t-s)^{-N/2}e^{-c\frac{|x-y|^2}{t-s}}
  \end{equation}
  exists for all $0\leq s<t\leq T$ and $x,y\in\Omega$. By the dominated
  convergence theorem
  \begin{displaymath}
    u(x,t)=\lim_{\lambda\to\infty}\int_\Omega
    k_\lambda(t,s,x,y)u_0(y)\,dy
    =\int_\Omega k_\infty(t,s,x,y)u_0(y)\,dy,
  \end{displaymath}
  so $U_\infty(t,s)$ has a kernel with a Gaussian estimate.  By
  splitting an arbitrary initial condition into its positive and
  negative part the above limit exists for every $u_0\in L^p(\Omega)$.

  Let now $1<p\leq q<\infty$. By H\"older's inequality
  \begin{displaymath}
    \begin{split}
      \|(U_\lambda(t,s)&-U_\infty(t,s))u_0\|_q\\
      &=\Bigl(\int_\Omega
      \Bigl|\int_\Omega
      \bigl(k_\lambda(x,y,t,s)-k_\infty(x,y,t,s)\bigr)u_0(y)\,dy
      \Bigr|^{q}\,dx \Bigr)^{1/q}\\
      &\leq\Bigl(\int_\Omega
      \Bigl(\int_\Omega
      \bigl(k_\lambda(x,y,t,s)-k_\infty(x,y,t,s)\bigr)^{p'}\,dy
      \Bigr)^{q/p'}\,dx \Bigr)^{1/q}\|u_0\|_p.
    \end{split}
  \end{displaymath}
  By \eqref{eq:uhkest}, \eqref{eq:kplim} and the dominated convergence
  theorem
  \begin{multline*}
    \|U_\lambda(t,s)-U_\infty(t,s)\|_{\mathcal L(L^p,L^q)}\\
    \leq\Bigl(\int_\Omega
    \Bigl(\int_\Omega
    \bigl(k_\lambda(x,y,t,s)-k_\infty(x,y,t,s)\bigr)^{p'}\,dy
    \Bigr)^{q/p'}\,dx \Bigr)^{1/q}
    \to 0
  \end{multline*}
  as $\lambda\to\infty$. Hence $U_\lambda(t,s)\to U_\infty(t,s)$ in
  $\mathcal L(L^p(\Omega),L^q(\Omega))$ whenever $1<p\leq q<\infty$. As
  the limit of compact operators is a compact operator, we conclude that
  $U_\infty(t,s)\in\mathcal L\bigl(L^p(\Omega),L^q(\Omega)\bigr)$ is
  compact.

  We next prove convergence of $u_0\to U_\lambda(\cdot\,,s)u_0$ as a
  linear operator with respect to the norm in $\mathcal
  L\bigl(L^p(\Omega),L^r((s,T),L^q(\Omega))\bigr)$ for suitable
  $r,p,q$. We already know from what we proved above that
  \begin{displaymath}
    \|U_\lambda(t,s)-U_\infty(t,s)\|_{\mathcal L(L^p,L^q)}\to 0
  \end{displaymath}
  for every $t\in(s,T]$ if $1<p\leq q<\infty$. We need to show that
  \begin{equation}
    \label{eq:U-norm-conv}
    \int_s^T\|U_\lambda(t,s)-U_\infty(t,s)\|_{\mathcal L(L^p,L^q)}^r\,dt\to 0
  \end{equation}
  as $\lambda\to\infty$. We deduce from \eqref{eq:hkest} that
  \begin{multline*}
    \|U_\lambda(t,s)-U_\infty(t,s)\|_{\mathcal L(L^p,L^q)}\\
    \leq\|U_\lambda(t,s)\|_{\mathcal L(L^p,L^q)}
    +\|U_\infty(t,s)\|_{\mathcal L(L^p,L^q)}
    \leq 2M(t-s)^{-\frac{N}{2}(\frac{1}{p}-\frac{1}{q})}e^{|\omega| T}
  \end{multline*}
  for all $0<s<t\leq T$ with constants $M$ and $\omega$ independent of
  $\lambda>0$. We note that
  \begin{displaymath}
    \int_s^T(t-s)^{-\frac{N}{2}(\frac{1}{p}-\frac{1}{q})r}
    e^{\omega (t-s)r}\,dt
    \leq e^{\omega rT}
    \int_s^T(t-s)^{-\frac{N}{2}(\frac{1}{p}-\frac{1}{q})r}\,dt
    <\infty
  \end{displaymath}
  if and only if \eqref{eq:pqr-cond} is satisfied. Hence,
  \eqref{eq:U-norm-conv} follows from the dominated convergence theorem.
\end{proof}
\begin{remark}
  The family $U_\infty(t,s)$, $0\leq s\leq t\leq T$, is not in general
  an evolution operator since in general $U_\infty(s,s$ is not the
  identity, but only a projection. In the extreme case where $m(x,t)>0$
  in $\Omega\times[0,T]$, then $U_\infty(t,s)=0$ is the zero
  operator. Hence we need conditions that guarantee that
  $U_\infty(\cdot\,,\cdot)$ is non-trivial.
\end{remark}
\begin{proposition}
  \label{prop:lower-bound}
  Suppose that $m\in L^\infty\bigl(\Omega\times (0,T)\bigr)$ and that
  there exists a non-empty open set $\Omega_0\subset\Omega$, $s_0\in
  [0,T)$ and $\varepsilon>0$ so that $m=0$ almost everywhere on
  $\Omega_0\times(s_0,s_0+\varepsilon)$. Then $U_\infty(t,s)\neq 0$ for
  $s_0<s\leq t<s_0+\varepsilon$. More precisely if $K(x,y,t,s)$ is the
  kernel of the evolution operator of the problem
  \begin{equation}
    \label{eq:hivp0}
    \begin{aligned}
      \frac{\partial u}{\partial t}+\mathcal A(t)u
      &=0
      &&\text{in $\Omega_0\times(s_0,s_0+\varepsilon)$,}\\
      u&=0
      &&\text{on $\partial \Omega_0\times(s_0,s_0+\varepsilon)$,}\\
      u(x,s_0)&=u_{s_0}(x)&&\text{in $\Omega_0$,}
    \end{aligned}
  \end{equation}
  then $k_\infty(x,y,t,s)\geq K(x,y,t,s)>0$ for all $x,y\in \Omega_0$ and
  all $s_0\leq s\leq t\leq s_0+\varepsilon$.
\end{proposition}
\begin{proof}
  Clearly the operators $\mathcal A(t)+\lambda m$ and $\mathcal A(t)$
  coincide on $\Omega_0\times(s_0,s_0+\varepsilon)$. Hence from
  \cite[Theorem~8.3]{daners:00:hke} we deduce that
  \begin{displaymath}
    k_\lambda(x,y,t,s)\geq K(x,y,t,s)>0
  \end{displaymath}
  for all $x,y\in \Omega_0$ and all $s_0\leq s\leq t\leq
  s_0+\varepsilon$. Here we also use that the kernel of the problem with
  Neumann or Robin boundary conditions dominates that of the problem with
  Dirichlet boundary conditions. Now the assertion of the theorem
  follows from \eqref{eq:kplim}.
\end{proof}
Using the evolution operator we can generalise the notion of solution of
\eqref{eq:pevp} for right hand sides not necessarily in
$L^2\bigl((0,T),V'\bigr)$.
\begin{definition}
  \label{def:mild-solution}
  Let $1\leq r,p\leq\infty$, $u_0\in L^p(\Omega)$ and $f\in
  L^r\bigl(0,T),L^p(\Omega)\bigr)$. We call
  \begin{equation}
    \label{eq:mild-solution}
    u(t)=U_{\lambda}(t,0)u_{0}
    +\int_0^tU_{\lambda}(t,\tau)f(\tau)\,d\tau,
  \end{equation}
  $t\in[0,T]$ a \emph{mild solution} of \eqref{eq:ivbp}. Likewise we
  call a $u$ a mild solution of the limit problem as $\lambda\to\infty$
  if
  \begin{equation}
    \label{eq:mild-solution-limit}
    u(t)=U_\infty(t,0)u_{0}
    +\int_0^tU_\infty(t,\tau)f(\tau)\,d\tau,
  \end{equation}
  for all $t\in[0,T]$.
\end{definition}
\begin{remark}
  \label{rem:mild-solution}
  By the Sobolev embedding theorem $V\hookrightarrow L^q\Omega)$ for
  $q\leq 2N/(N-2)$ if $N\geq 3$ and $q<\infty$ if $N=2$. Hence
  $L^p(\Omega)\hookrightarrow V'$ for $p\geq 2N/(N+2)$ if $N\geq 3$ and
  $p>1$ if $N=2$. Thus, if $r\geq 2$, then
  \begin{displaymath}
    L^r\bigl((0,T),L^p(\Omega)\bigr)
    \hookrightarrow L^2\bigl((0,T),V'\bigr)
  \end{displaymath}
  for $p\geq 2N/(N+2)$ if $N\geq 3$ and $p>1$ if $N=2$. The above
  embedding always holds if $N=1$ and $1\leq p\leq\infty$. In these
  cases every mild solution of \eqref{eq:pevp} is a weak solution of
  \eqref{eq:pevp}.
\end{remark}
Now that we know that the limit problem is non-trivial in general we
strengthen some results from Theorem~\ref{thm:compact}.
\begin{theorem}
  \label{thm:compact-Lp}
  Assume that $m$ satisfies \eqref{eq:m-top-reg}. Suppose that
  $1<p\leq q\leq\infty$ and $1<r<\infty$ such that
  \begin{equation}
    \label{eq:pqrr-cond}
    \frac{N}{2}\Bigl(\frac{1}{p}-\frac{1}{q}\Bigr)
    <\min\Bigl\{\frac{1}{r},1-\frac{1}{r}\Bigr\},
  \end{equation}
  Assume that $u_{0n}\rightharpoonup u_0$ weakly in $L^p(\Omega)$ and
  that $f_n\to f$ in $L^r\bigl((0,T),L^p(\Omega)\bigr)$.  Let $u_n$ be
  the mild solution of \eqref{eq:ivbp} with $\lambda$ replaced by
  $\lambda_n$, $f$ replaced by $f_n$ and $u_0$ replaced by
  $u_{0n}$. Finally suppose that $\lambda_n\to\infty$.  Then $u_n(t)\to
  u(t)$ in $L^q(\Omega)$ for all $t\in (0,T]$ and $u$
  satisfies \eqref{eq:mild-solution-limit}. Moreover, $u_n\to u$ in
  $L^r\bigl((0,T),L^q(\Omega)\bigr)$.
\end{theorem}
\begin{proof}
  We know that
  \begin{equation}
    \label{eq:mild-solution-n}
    u_n(t)=U_{\lambda_n}(t,0)u_{0n}
    +\int_0^tU_{\lambda_n}(t,\tau)f_n(\tau)\,d\tau
  \end{equation}
  for all $t\in(0,T]$.  As $(u_{0n})$ is weakly convergent in
  $L^p(\Omega)$ there exists $c_1>0$ such that $\|u_{0n}\|_p\leq c_1$
  for all $n\in\mathbb N$. Moreover, since $U_\infty(t,0)\in\mathcal
  L\bigl(L^p(\Omega),L^q(\Omega)\bigr)$ is compact and
  $U_{\lambda_n}(t,0)\to U_\infty(t,0)$ in $\mathcal
  L\bigl(L^p(\Omega),L^q(\Omega)\bigr)$ by Theorem~\ref{thm:U-limit} we
  see that
  \begin{multline*}
    \|U_{\lambda_n}(t,0)u_{0n}-U_\infty(t,0)u_0\|_q\\
    \leq\|U_{\lambda_n}(t,0)-U_\infty(t,0)\|_{\mathcal L(L^p,L^q)}\|u_{0n}\|_p
    +\|U_\infty(t,0)(u_{0n}-u_0)\|_q
    \to 0
  \end{multline*}
  for every $t\in(0,T]$ as $n\to\infty$. Using the uniform kernel
  estimates from Theorem~\ref{thm:unif-est} we see that
  \begin{displaymath}
    \|U_{\lambda_n}(t,0)u_{0n}-U_\infty(t,0)u_0\|_q
    \leq 2e^{|\omega| T}Mc_1t^{-\frac{N}{2}(\frac{1}{p}-\frac{1}{q})}
  \end{displaymath}
  for all $t\in(0,T]$. As $t^{-\frac{N}{2}(\frac{1}{p}-\frac{1}{q})r}$
  is integrable on $(0,T)$ by \eqref{eq:pqrr-cond}, the dominated
  convergence theorem implies that $U_{\lambda_n}(\cdot\,,0)u_{0n}\to
  U_\infty(\cdot\,,0)u_0$ in $L^r\bigl((0,T),L^q(\Omega)\bigr)$.

  We next deal with the integral term in
  \eqref{eq:mild-solution-n}. Using H\"older's inequality,
  \begin{equation}
    \label{eq:conv-inhom-part}
    \begin{split}
      \int_0^t\|&U_{\lambda_n}(t,\tau)f_n(\tau)
      -U_\infty(t,\tau)f(\tau)\|_q\,d\tau\\
      &\leq\int_0^t\|U_{\lambda_n}(t,\tau)-U_\infty(t,\tau)\|_{\mathcal
        L(L^p,L^q)}
      \|f_n(\tau)\|_p\,d\tau\\
      &\qquad\qquad+\int_0^t\|U_\infty(t,\tau)\|_{\mathcal
        L(L^p,L^q)}\|f_n(\tau)-f(\tau)\|_q\,d\tau\\
      &\leq\Bigl(\int_0^t\|U_{\lambda_n}(t,\tau)-U_\infty(t,\tau)\|_{\mathcal
        L(L^p,L^q)}^{\frac{r}{r-1}}\tau\Bigr)^{1-\frac{1}{r}}\,d\tau
      \|f_n\|_{L^r((0,T)L^p)}\\
      &\qquad\qquad+\Bigl(\int_0^t\|U_{\lambda_n}(t,\tau)\|_{\mathcal
        L(L^p,L^q)}^{\frac{r}{r-1}}\,d\tau\Bigr)^{1-\frac{1}{r}}
      \|f_n-f\|_{L^r((0,T)L^p)}
    \end{split}
  \end{equation}
  for all $t\in[0,T]$. By again using the uniform kernel estimates
  \begin{displaymath}
    \|U_{\lambda_n}(t,\tau)-U_\infty(t,\tau)\|_{\mathcal L(L^p,L^q)}
    \leq 2e^{|\omega|T}M(t-\tau)^{-\frac{N}{2}(\frac{1}{p}-\frac{1}{q})}
  \end{displaymath}
  for all $0\leq \tau<t\leq T$. It follows from \eqref{eq:pqrr-cond}
  that
  \begin{displaymath}
    \eta:=\frac{N}{2}\Bigl(\frac{1}{p}-\frac{1}{q}\Bigr)\frac{r}{r-1}<1.
  \end{displaymath}
  and hence $(t-\tau)^{-\eta}$ is integrable on $(0,t)$. As $f_n\to f$
  in $L^r\bigl((0,T),L_p(\Omega)\bigr)$ we conclude from
  Theorem~\ref{thm:U-limit} and \eqref{eq:conv-inhom-part} that
  \begin{displaymath}
    \int_0^tU_{\lambda_n}(t,\tau)f_n(\tau)\,d\tau
    \to \int_0^tU_\infty(t,\tau)f(\tau)\,d\tau
  \end{displaymath}
  in $L^q(\Omega)$ for all $t\in (0,T]$. It also follows that
  \begin{displaymath}
    \Bigl\|\int_0^tU_{\lambda_n}(t,\tau)f_n(\tau)\,d\tau\Bigr\|_q
    \leq e^{|\omega|T}C t^{1-\eta}
  \end{displaymath}
  for all $t\in [0,T]$, where $C$ is a constant independent of $n$ and
  $1-\eta>0$. In particular, the integral part in
  \eqref{eq:mild-solution-n} converges in
  $L^r\bigl((0,T),L^q(\Omega)\bigr)$ as well.
\end{proof}
In the above theorem we have excluded the case $r,q=\infty$, that is,
uniform convergence. The next theorem shows local convergence in a space
of (locally) H\"older continuous functions on
$D_m\cap(\Omega\times(\varepsilon,T])$ for every $\varepsilon\in(0,T)$. 
\begin{theorem}
  \label{thm:compact-holder}
  Assume that $m$ satisfies \eqref{eq:m-top-reg}. Suppose that
  $N/2<p\leq\infty$ and $2\leq r<\infty$ such that
  \begin{equation}
    \label{eq:rp-condition}
    \frac{N}{2p}+\frac{1}{r}<1,
  \end{equation}
  Assume that $u_{0n}\rightharpoonup u_0$ weakly in $L^p(\Omega)$ and
  that $f_n\to f$ in $L^r\bigl((0,T),L_p(\Omega)\bigr)$.  Let $u_n$ be
  the mild solution of \eqref{eq:ivbp} with $\lambda$ replaced by
  $\lambda_n$, $f$ replaced by $f_n$ and $u_0$ replaced by
  $u_{0n}$. Finally suppose that $\lambda_n\to\infty$. 
  Then for every $\varepsilon\in(0,T)$ and every compact subset
  $K\subseteq D_m\cap(\Omega\times[\varepsilon,T])$ there exists
  $\beta\in(0,1)$ such that $u_n\to u$ in $C^\beta(K)$.
\end{theorem}
\begin{proof}
  First note that \eqref{eq:rp-condition} implies that $f_n\in
  L^2\bigl((0,T),V'\bigr)$, and that the sequence $(f_n)$ is bounded in
  that space; see Remark~\ref{rem:mild-solution}. Hence, by
  Theorem~\ref{thm:compact} $u_n\rightharpoonup u$ weakly in
  $L^2\bigl((0,T),V\bigr)$. By Corollary~\ref{cor:limit-problem} $u$ is
  a weak solution of \eqref{eq:ivbp-limit}. The weak solution of
  \eqref{eq:ivbp} is given by
  \begin{displaymath}
    u_n(t)=U_{\lambda_n}(t,0)u_{0n}+\int_0^tU_{\lambda_n}(t,\tau)f_n(\tau)\,d\tau.
  \end{displaymath}
  Using the uniform kernel estimates from Theorem~\ref{thm:unif-est} we
  see that
  \begin{equation}
    \label{eq:U-est-1}
    \begin{aligned}
      \|U_{\lambda_n}(t,0)u_{0n}\|_\infty
      &\leq Mt^{-\frac{N}{2p}}e^{\omega t}\|u_{0n}\|_p\\
      \|U_{\lambda_n}(t,\tau)f_n(\tau)\|_\infty
      &\leq M(t-\tau)^{-\frac{N}{2p}}e^{\omega(t-s)}\|f_n(\tau)\|_p
    \end{aligned}
  \end{equation}
  for all $n\in\mathbb N$ and all $0<\tau<t\leq T$. Hence by H\"older's
  inequality
  \begin{equation}
    \label{eq:uniform-bounded}
    \Bigl\|\int_0^tU_{\lambda_n}(t,\tau)f_n(\tau)\,d\tau\Bigr\|_\infty
    \leq Me^{|\omega|T}
    \Bigl(\int_0^t(t-\tau)^{-\frac{N}{2p}\frac{r}{r-1}}\,d\tau\Bigr)
    ^{1-\frac{1}{r}}\|f_n\|_{L^r((0,T),L^p)}
  \end{equation}
  The second integral in \eqref{eq:uniform-bounded} is finite if and
  only if \eqref{eq:rp-condition} holds. Putting everything together we
  see that the sequence $(u_n)$ is bounded in
  $L^\infty\bigl(\Omega\times(\varepsilon,T]\bigr)$ for every
  $\varepsilon>0$. Since $u_n$ is a solution of \eqref{eq:ivbp-limit}
  with $f$ replaced by $f_n$ we conclude from
  \cite[Theorem~4]{aronson:67:lbs} that for $\varepsilon\in(0,T)$ and
  every compact subset $K\subseteq D_m\cap([\varepsilon,T]\times\Omega)$
  there exists $\gamma\in(0,1)$ such that $u_n$ is bounded in
  $C^\gamma(K)$. As we know that $u_n\rightharpoonup u$ weakly in
  $L^2\bigl(\Omega\times(0,T)\bigr)$, we conclude that $u_n\to u$ in
  $C^\beta(K)$ for $\beta\in(0,\gamma)$. Here we use that H\"older
  spaces with different exponents embed compactly.
\end{proof}
\begin{remark}
  \label{rem:compact-holder}
  If we strengthen the regularity assumptions on the coefficients of
  $(\mathcal A(t),\mathcal B(t))$, $m$ and $f_n$ we obtain (local)
  convergence in $D_m$ in stronger norms. In particular, assume that the
  the coefficients of the diffusion matrix $D$ and the vector field $a$
  in \eqref{eq:A} are in $C^{1+\beta,\beta/2}$ and $b,c,b_0,m,f_n$ are
  in $C^{\beta,\beta/2}$ for some $\beta\in(0,1)$. Then the Schauder
  theory in \cite[Theorem~VI.10.1]{ladyzenskaya:68:lqp} or
  \cite[Theorem~3.4.9]{lieberman:96:sop} shows that on every compact
  subset $K\subseteq D_m\cap(\Omega\times[\varepsilon,T])$ there exists
  $\gamma\in(0,1)$ such that $u_n$ is bounded in
  $C^{2+\gamma,1+\gamma/2}(K)$. Hence $u_n\to u$ in
  $C^{2+\beta,1+\beta/2}(K)$ for $\beta\in(0,\gamma)$. Convergence in
  $C^{\beta,\beta/2}$ may also be true up to the boundary depending on
  the regularity of $D_m$, in particular if $D_m$ contains parabolic
  cylinders with sufficiently smooth boundary such as the situation
  considered in \cite{du:12:ple} corresponding to the example on the
  right in Figure~\ref{fig:D_m}; see also Example~\ref{ex:du-peng}.
\end{remark}
Based on Proposition~\ref{prop:lower-bound} we show that $U_\infty(t,s)$
is has some nice properties for all $0\leq s<t\leq T$ if $m$ satisfies
certain conditions.
\begin{assumption}
  \label{ass:m-good}
  Let $m\in L^\infty(\Omega\times [0,T])$ and assume that the support of
  $m$ is topologically regular, that is, \eqref{eq:m-top-reg} is
  satisfied.  We define the sets $D_m$ and $\Omega_t$ as in
  \eqref{eq:D_m} and \eqref{eq:Omega_t} respectively.  Assume that
  $\Omega_t\neq\emptyset$ for every $t\in [0,T]$. Suppose that for every
  pair of points $y\in\Omega_0$ and $x\in\Omega_t$ with $t\in(0,T]$
  there exist a continuous function $\varphi\colon[0,t]\to\Omega$ with
  $\varphi(0)=y$, $\varphi(1)=x$ and such that $(\varphi(\tau),\tau)\in
  D_m$ for all $\tau\in[0,t]$; see Figure~\ref{fig:D_m}.
\end{assumption}
\begin{remark}
  \label{rem:maxprinc}
  (a) The condition about the existence of the curve $\varphi$ in
  Assumption~\ref{ass:m-good} is related to the condition on
  non-cylindrical regions in the parabolic maximum principle. The
  condition for the validity of the maximum principle is that the point
  is to be reached by a continuous path that only goes ``horizontal'' or
  ``upwards'', that is, ``forward'' in time; see
  \cite[p169]{protter:84:mpd} or \cite{friedman:58:rmp},
  where also a counter example is shown if the condition is violated. As
  a consequence the limit problem is well behaved in the sense that the
  parabolic maximum principle is valid for the non-cylindrical domain
  $D_m$ and hence there are uniqueness theorems.

  (b) Our condition also guarantees that $D_m$ is connected. If $D_m$ is
  not connected, we apply our arguments to every connected
  component. Examples are shown in Figure~\ref{fig:D_m}.

  (c) The diagram on the right in Figure~\ref{fig:D_m} is the special
  situation considered in \cite{alvarez:14:qac,du:12:ple}, where $T^*$
  is as in these references.
\end{remark}
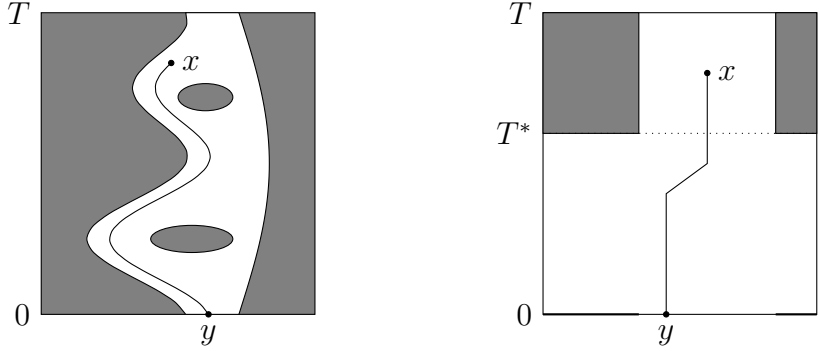
\begin{figure}[ht]
  \centering
  \begin{tikzpicture}
    \pgfmathsetmacro\T{4}
    \pgfmathsetmacro\w{1.8}
    \draw[fill=gray] (-\w,0) node[left] {$0$} -- (\w,0) -- (\w,\T) --
    (-\w,\T) node[left] {$T$} -- cycle;
    \begin{scope}[domain=0:180,smooth]
      \draw[fill=white] plot ({.4*sin(\x)+.8},{\T*\x/180})
      -- plot ({-.4+.5*cos(4*\x)+.3*sin(2*\x)},{\T*(180-\x)/180}) -- cycle;%
    \end{scope}
    \draw[fill=gray] (.1*\w,.25*\T) circle [x radius={.3*\w},y
    radius=.1*\w];%
    \draw[fill=gray] (.2*\w,.72*\T) circle [x radius={.2*\w},y
    radius=.1*\w];%
    \draw[domain=30:180,smooth] plot
    ({-.1+.5*cos(4*\x)+.3*sin(2*\x)},{\T*(180-\x)/180});%
    \draw[fill=black] (.4,0) circle (1pt) node[below] {$y$};%
    \draw[fill=black] ({-.1+.5*cos(120)+.3*sin(60)},{\T*150/180}) circle
    (1pt) node[right] {$x$};%
  \end{tikzpicture}
  \hfil
  \begin{tikzpicture}
    \pgfmathsetmacro\T{4}
    \pgfmathsetmacro\w{1.8}
    \draw (-\w,0) node[left] {$0$} -- (\w,0) -- (\w,\T) --
    (-\w,\T) node[left] {$T$} -- cycle;
    \begin{scope}
      \draw[fill=gray] (-\w,.6*\T) -| (-.3*\w,\T) -- (-\w,\T) -- cycle;%
      \draw[fill=gray] (\w,.6*\T) -| (.7*\w,\T) -- (\w,\T) -- cycle;%
    \end{scope}
    \draw[thick] (-\w,0) -- (-.3*\w,0);%
    \draw[thick] (\w,0) -- (.7*\w,0);%
    \draw[dotted] (-\w,.6*\T) node[left] {$T^*$} -- (\w,.6*\T);%
    \draw (-.1*\w,0) -- (-.1*\w,.4*\T) -- (.2*\w,.5*\T) --(.2*\w,.8*\T);%
    \draw[fill=black] (-.1*\w,0) circle (1pt) node[below] {$y$};%
    \draw[fill=black] (.2*\w,.8*\T) circle (1pt) node[right] {$x$};%
  \end{tikzpicture}
  \caption{Parabolic cylinder with $\supp m$ (shaded) and $D_m$ with
    $0\leq t\leq T$.}
  \label{fig:D_m}
\end{figure}
\begin{theorem}
  \label{thm:U-limit-non-trivial}
  Suppose that $m$ satisfies Assumption~\ref{ass:m-good} and let
  $k_\infty$ be the kernel of the limit evolution system $U_\infty$ as
  in Theorem~\ref{thm:U-limit}. Then $k_\infty(x,y,t,0)>0$ for all
  $t\in(0,T]$ and all $(x,y)\in\Omega_t\times\Omega_0$.
\end{theorem}
\begin{proof}
  Fix $(y,0),(x,t)\in D_m$ with $0<t\leq T$ and let
  $\varphi\colon[0,t]\to\Omega$ be as in Assumption~\ref{ass:m-good}. We
  consider the set
  \begin{displaymath}
    I:=\bigl\{s\in(0,t]\colon 
    \text{$k_\infty(\varphi(\tau),y,\tau,0)>0$
      for all $\tau\in[0,s]$}\bigr\}.
  \end{displaymath}
  We need to show that $I=(0,t]$. Because $(0,t]$ is connected it is
  sufficient to show that $I$ is non-empty, open and closed in
  $(0,t]$. To do so we use the fact that the function
  \begin{equation}
    \label{eq:kernel-continuous}
    (0,t]\to[0,\infty),\,\tau\mapsto k_\infty(\varphi(\tau),y,\tau,0)
    \qquad\text{is continuous.}
  \end{equation}
  Further note that if $s\in I$, then $(0,s]\subseteq I$ by definition
  of $I$.

  We first show that $I$ is non-empty. As $(y,0)\in D_m$ and $D_m$ is
  open there exists an open neighbourhood $V_0\subseteq\Omega$ of $y$
  and an interval $[0,s_0]$ such that $V_0\times [0,s_0]\subseteq D_m$.
  Proposition~\ref{prop:lower-bound} implies that
  $k_\infty(z,y,\tau,0)>0$ for all $(z,\tau)\in V_0\times J_0$.  By
  \eqref{eq:kernel-continuous} there exists $\tau_0\in (0,s_0]$ such
  that $\varphi(\tau)\in V$ for all $\tau\in(0,\tau_0]$ and hence
  $k_\infty(\varphi(\tau),y,\tau,0)>0$ for all $\tau\in
  (0,\tau_0]$. Hence, $\tau_0\in I$ and so $I\neq\emptyset$.

  We next show that $I$ is open. If $s\in I$, then
  $k(\varphi(\tau),y,\tau,0)>0$ for all $\tau\in(0,s]$. In particular,
  $k(\varphi(s),y,s,0)>0$. By \eqref{eq:kernel-continuous} there exists
  $s_1>s$ so that $k(\varphi(\tau),y,\tau,0)>0$ for all
  $\tau\in(0,s_1]$. Hence $I$ is open.

  We finally show that $I$ is closed. If $s>0$ is in the closure of $I$,
  then $k_\infty(\varphi(\tau),y,\tau,0)>0$ for all
  $\tau\in(0,s)$. Because $D_m$ is open there exists a non-empty open
  set $V\subseteq\Omega$ and an open interval $J\subseteq\mathbb R$ such
  that $\bigl(\varphi(s),s\bigr)\in V\times J\subseteq D_m$. Now
  Proposition~\ref{prop:lower-bound} implies that
  $k_\infty(z,w,s,\tau)>0$ for all $z,w\in V$ and all $\tau\in J$ with
  $\tau<s$. Due to \eqref{eq:kernel-continuous} we can choose $\tau_0\in
  J$ with $\tau_0<s$ such that $\varphi(\tau_0)\in V$. Then, by
  \eqref{eq:U-limit-property}
  \begin{equation}
    \label{eq:k-non-trivial}
    k_\infty(x,y,s,0)
    =\int_{\Omega_s}k_\infty(x,z,s,\tau_0)k_\infty(z,y,\tau_0,0)\,dz.
  \end{equation}
  We have chosen $V$ and $J$ such that $k_\infty(x,z,s,\tau_0)>0$ for
  all $z\in V$. In particular $k_\infty(x,z,s,\tau_0)>0$ for $z$ in a
  neighbourhood of $\varphi(\tau_0)$. By the continuity of
  $k_\infty(x,z,s,\tau_0)$ as a function of $z$ we also deduce that
  $k_\infty(x,z,s,\tau_0)>0$ for all $z$ in a neighbourhood of
  $\varphi(\tau_0)$. As $k_\infty\geq 0$ we conclude from
  \eqref{eq:k-non-trivial} that $k_\infty(x,y,\tau_0,0)>0$. Hence, $s\in
  I$ and thus $I$ is closed.
\end{proof}

We complete this section by reviewing the special case of $m$ treated in
\cite{du:12:ple}.

\begin{example}
  \label{ex:du-peng}
  The special case considered in \cite{du:12:ple} is $m$ of the form
  $m(x,t)=p(x)q(t)$ with $\supp(p)=\Omega\setminus U_0$ for some
  non-empty open set $U_0$ and $\supp q=[T^*,T]$ for some $T^*\in(0,T)$,
  or the slightly more general situation given in
  \cite[condition~(3.2)]{du:12:ple}. The situation is depicted in
  Figure~\ref{fig:D_m} on the right. We only assume that $m\in
  L^\infty\bigl(\Omega\times(0,T)\bigr)$. The set $D_m$ consists of two
  cylindrical regions: $\Omega\times(0,T^*)$ and $U_0\times
  (T^*,T)$. Let now $r\geq 2$ and $f\in
  L^r\bigl((0,T),L^p(\Omega)\bigr)$ with $p$ satisfying
  \eqref{eq:rp-condition}. If $\Omega$ and $U_0$ are regular enough,
  such as in \cite{du:12:ple}, then standard regularity theory for
  parabolic equations imply that the convergence is actually uniform in
  $\Omega\times[\varepsilon,T^*-\varepsilon]$ and (or) in
  $U_0\times[T^*+\varepsilon,T]$. According to
  \cite[Theorem~III.10.1]{ladyzenskaya:68:lqp} the H\"older estimates in
  Theorem~\ref{thm:compact-holder} are not just local, but global in the
  above cylinders for every $\varepsilon>0$ sufficiently small. A
  sufficient condition is that $\Omega$ and (or) $U_0$ satisfies a
  uniform exterior cone condition. Such a condition is satisfied in
  \cite{du:12:ple}.
\end{example}

\section{The periodic-parabolic eigenvalue problem}
\label{sec:periodic}
In this section we study the principal eigenvalue of the
periodic-parabolic eigenvalue problem \eqref{eq:pevp} as a function of
$\lambda$. In particular we assume throughout that the coefficients of
$(\mathcal A(t),\mathcal B(t))$ as well as the weight function $m(x,t)$
are $T$-periodic as a function of time $t\in\mathbb R$.

It is well known that there is a one-to-one correspondence between the
real eigenvalues and corresponding eigenfunctions of \eqref{eq:pevp} and
the positive eigenvalues of $U_\lambda(T,0)$ and their eigenfunctions;
see \cite[Prop~14.4]{hess:91:ppb}. Indeed, the following lemma is easily
checked, see also \cite{daners:92:aee,daners:00:epp,hess:91:ppb}.

\begin{lemma}
  \label{lem:1}
  Let the assumptions of Section~\ref{sec:weak-solutions} be satisfied. Then
  $\beta(\lambda)\in\mathbb R$ is an eigenvalue of $U_\lambda(T,0)$ with
  eigenfunction $w_\lambda\in L^2(\Omega)$ if and only if
  \begin{displaymath}
    \mu(\lambda):=-\frac{1}{T}\log\bigl(\beta(\lambda)\bigr)
  \end{displaymath}
  is a periodic-parabolic eigenvalue of \eqref{eq:pevp} with
  $T$-periodic eigenfunction $u_\lambda\in C\bigl(\mathbb
  R,L^2(\Omega)\bigr)$ given by
  \begin{displaymath}
    u_\lambda(t):=e^{\mu(\lambda)t}U_\lambda(t,0)w_\lambda
  \end{displaymath}
  for all $t\in\mathbb R$.
\end{lemma}
We next want show that under Assumption~\ref{ass:m-good} the limit
problem as $\lambda\to\infty$ has a periodic-parabolic principal
eigenvalue $\mu_\infty$ that can be obtained as the limit of
$\mu(\lambda)$. We also show that the corresponding eigenfunctions can
be chosen so that they converge in $L^p(\Omega\times(0,T))$ for $1\leq
p<\infty$ and in $C^\beta(D_m)$, that is, locally in a H\"older norm. As
mentioned already the theorem generalises and simplifies a results in
\cite[Theorem~3.3 and~3.4]{du:12:ple}, where a very special case of
Assumption~\ref{ass:m-good} is covered.
\begin{theorem}
  \label{thm:pp-eigenvalue}
  Suppose $m\in L^\infty\bigl(\Omega\times\mathbb R\bigr)$ is
  $T$-periodic and satisfies Assumption~\ref{ass:m-good}. Let
  $\mu(\lambda)$ be the principal eigenvalue of the periodic-parabolic
  problem \eqref{eq:pevp}. Then $\mu(\lambda)$ is an increasing function
  of $\lambda>0$ and
  \begin{equation}
    \label{eq:mu-infty}
    \mu_\infty:=\lim_{\lambda\to\infty}\mu(\lambda)\in\mathbb R
  \end{equation}
  exists. Furthermore, we can choose eigenfunctions $u_\lambda\in
  L^\infty\bigl(\Omega\times(0,T)\bigr)$ of \eqref{eq:pevp} such that
  \begin{equation}
    \label{eq:pp-ev}
    u_\infty(t)
    =\lim_{\lambda\to\infty}u_\lambda(t)
    =e^{\mu_\infty t}U_\lambda(t,0)u_\infty(0)
  \end{equation}
  in $L^q\bigl(\Omega)$ for all $t\in\mathbb R$ whenever $1\leq
  q<\infty$. Moreover, for every compact subset $K\subseteq
  D_m\cap(\Omega\times[0,T])$ there exists $\beta\in(0,1)$ such that
  $u_\lambda\to u_\infty$ in $C^\beta(K)$ as $\lambda\to\infty$.
  Finally, $\mu_\infty$ is the unique principal eigenvalue of
  \begin{equation}
    \label{eq:evp-limit}
    \begin{aligned}
      \frac{\partial u}{\partial t}+\mathcal A(t)u
      &=\mu_\infty u
      &&\text{in $D_m$,}\\
      \mathcal B(t)u&=0
      &&\text{on $\partial D_m\cap\bigl(\partial\Omega\times(0,T)\bigr)$,}\\
      u(x,0)&=u(x,T)&&\text{on $\Omega_0$,}
    \end{aligned}
  \end{equation}
  and $u_\infty$ is the unique positive eigenfunction up to scalar
  multiples. If $\Omega_t$ is regular for all $t\in[0,T]$ as in
  Corollary~\ref{cor:limit-problem}, then $u_\infty$ satisfies Dirichlet
  boundary conditions on $\partial\Omega_t\cap\Omega$ for almost all
  $t\in\mathbb R$.
\end{theorem}
\begin{proof}
  If $\mu(\lambda)$ is the principal eigenvalue of  \eqref{eq:pevp},
  then by Lemma~\ref{lem:1} we have
  \begin{displaymath}
    r(\lambda):=\spr\bigl(U_\lambda(T,0)\bigr)=e^{-\mu(\lambda)T}.
  \end{displaymath}
  We have proved that $U_\lambda(t,s)$ is decreasing as a function of
  $\lambda$ and therefore standard theory of positive operators on the
  Banach lattice $L^2(\Omega)$ implies that $r(\lambda)$ is decreasing
  in $\lambda$. Hence $\mu(\lambda)$ is an increasing function of
  $\lambda$.

  We know from Theorem~\ref{thm:U-limit} that $U_\lambda(T,0)\to
  U_\infty(T,0)$ in $\mathcal L\bigl(L^2(\Omega)\bigr)$, and that
  $U_\infty(T,0)$ is compact.  We have proved in
  Theorem~\ref{thm:U-limit-non-trivial} that $U_\infty(T,0)$ has a
  kernel $k_\infty$ with the property that $k_\infty(x,y,T,0)>0$ for all
  $(x,y)\in\Omega_T\times\Omega_0$. Hence, if $w\in L^2(\Omega_0)$ is
  non-negative with $w>0$ on a set of positive measure, then
  \begin{displaymath}
    \bigl(U_\infty(T,0)w\bigr)(x)
    =\int_{\Omega_0}k_\infty(x,y,T,0)w(y)\,dy>0
  \end{displaymath}
  for all $x\in\Omega_T$.  By the $T$-periodicity we have
  $\Omega_T=\Omega_0$ and therefore $\bigl(U_\infty(T,0)w\bigr)(x)>0$
  for all $x\in\Omega_0$ if $w>0$ on $\Omega_0$. Hence, the restriction
  of $U_\infty(T,0)$ to $L^2(\Omega_0)$ is a compact positive and
  irreducible operator on $L^2(\Omega_0)$. Therefore, the spectral
  radius $r_\infty:=\spr\bigl(U_\infty(T,0)\bigr)$ is an algebraically
  simple eigenvalue of $U_\infty(T,0)$ and $r_\infty>0$ by a
  generalisation of the Krein-Rutman theorem due to
  \cite{depagter:86:ico}.

  By using perturbation results involving extensions and restrictions to
  sub-domains such as \cite[Section~4.3]{daners:08:dpl}, we have
  \begin{equation}
    \label{eq:1}
    r_\infty=\lim_{\lambda\to\infty}r(\lambda)
    =\spr\bigl(U_\infty(T,0)\bigr).
  \end{equation}
  Moreover, we can choose eigenfunctions $w_\lambda>0$ of
  $U_\lambda(T,0)$ such that $w_\lambda\to w_\infty$ in $L^2(\Omega)$ as
  $\lambda\to\infty$.  In particular, $w_\infty$ is an eigenfunction of
  $U_\infty(T,0)$ corresponding to the eigenvalue $r_\infty$, and
  \begin{displaymath}
    \mu_\infty=-\frac{1}{T}\log r_\infty<\infty.
  \end{displaymath}
  We know from Lemma~\ref{lem:1} that 
  \begin{displaymath}
    u_\lambda(t)=e^{\mu(\lambda)t}U_\lambda(t,0)w_\lambda,\qquad
    t\in\mathbb R
  \end{displaymath}
  is a positive periodic-parabolic eigenfunction of \eqref{eq:pevp}.  It
  follows from Theorem~\ref{thm:compact-Lp} that
  \begin{displaymath}
    u_\lambda(t)\to u_\infty(t):=e^{\mu_\infty t}U(t,0)w_\infty
  \end{displaymath}
  in $L^2(\Omega)$ for all $t>0$, and hence by the $T$-periodicity for
  all $t\in\mathbb R$. The above argument also implies the uniqueness of
  the periodic-parabolic eigenvalue and eigenfunction up to scalar
  multiples.

  Applying an estimate similar to \eqref{eq:U-est-1} as well as the fact
  that $\mu_\infty\geq\mu(\lambda)$ we see that
  \begin{displaymath}
    \|u_\lambda(t)\|_\infty
    \leq e^{2|\omega|T}e^{\mu_\infty t}T^{-N/4}\|w_\lambda\|_2
  \end{displaymath}
  for all $t\in[T,2T]$. As $w_\lambda$ is bounded in $L^2(\Omega)$ and
  $u_\lambda$ is $T$-periodic it follows that the family of
  periodic-parabolic eigenfunctions $(u_\lambda)$ is bounded in
  $L^\infty\bigl(\mathbb R\times\Omega\bigr)$. Let $K\subseteq
  D_m\cap(\Omega\times[0,T])$ and consider the compact set
  \begin{displaymath}
    \tilde K
    :=\{(x,t+T)\colon (x,t)\in K\}\subseteq D_m\cap(\Omega\times[T,2T]).
  \end{displaymath}
  By Theorem~\ref{thm:compact-holder} there exists $\beta\in(0,1)$ such
  that $u_\lambda\to u_\infty$ in $C^\beta(\tilde K)$ as
  $\lambda\to\infty$. By the $T$-periodicity we also have $u_\lambda\to
  u_\infty$ in $C^\beta(K)$ as $\lambda\to\infty$.  Finally,
  Theorem~\ref{thm:U-limit-non-trivial} and periodicity it is strictly
  positive on $D_m$.
\end{proof}
\begin{remark}
  \label{rem:holder-cont}
  Under stronger assumptions on the regularity of the coefficients such
  as those in Remark~\ref{rem:compact-holder}, for every compact subset
  $K\subseteq D_m$ we have $u_\lambda\to u_\infty$ in
  $C^{2+\beta,1+\beta/2}(K)$ for some $\beta\in(0,1)$. Due to
  Corollary~\ref{cor:limit-problem} we can also deduce H\"older
  regularity of $u_\infty$ up to some parts of $\partial D_m$, but
  depending on geometry and smoothness assumptions on $D_m$. This
  recovers the regularity result in \cite[Theorem~3.3]{du:12:ple}.
\end{remark}
We next make some comparisons to the earlier work in \cite{du:12:ple}.
\begin{example}
  \label{ex:du-peng-pp}
  If we combine the comments in Example~\ref{ex:du-peng} with
  Theorem~\ref{thm:pp-eigenvalue} we can strengthen the convergence
  result in \cite{du:12:ple}. Rather than having local uniform
  convergence of the periodic-parabolic eigenfunction $u_\lambda$ in
  $D_m$ we have \emph{global} uniform convergence of $u_\lambda$ in the
  cylinders $[\varepsilon,T^*-\varepsilon]\times\Omega$ and (or) in
  $[T^*+\varepsilon,T]\times U_0$, depending on the regularity
  assumptions on $\Omega$ and (or) $U_0$. These conditions are satisfied
  in \cite{du:12:ple}.
\end{example}

\section{Optimality of the conditions on the weight function}
\label{sec:optimal}
We note that Assumption~\ref{ass:m-good} is not necessary to guarantee a
solution to the limit problem. Indeed, suppose there exists
$T$-periodic $\tilde{m}\in L^\infty(\Omega\times\mathbb R)$ satisfying
Assumption~\ref{ass:m-good} such that $m\leq\tilde{m}$. Then
Theorem~\ref{thm:pp-eigenvalue} applied to the problem with $\tilde{m}$,
along with a similar argument to the proof of \eqref{eq:uhkest} in
Theorem~\ref{thm:unif-est}, implies that $\mu_\infty$ is an eigenvalue
of the limit problem. In particular, the vanishing set $D_m$ need only
satisfy the conditions of Assumption~\ref{ass:m-good} on a nonempty open
subset. However, in this case we cannot guarantee that the eigenfunction
of the limit problem is strictly positive, nor that it is
unique. Non-uniqueness can occur if the set $D_m$ has for instance two
connected components, both satisfying Assumption~\ref{ass:m-good}.

Nevertheless, we show now that the condition just described cannot be
omitted. In an extreme case we could consider a situation where $m>0$ on
a set $\Omega\times I$, where $I\subseteq(0,T)$ is a non-trivial
interval. Then clearly $U_\infty(T,0)=0$ and there is no
periodic-parabolic eigenvalue and eigenfunction associated with the
limit problem.

Even if $D_m$ is path-connected, such a situation can arise. In the
following example, the set $D_m$ is path-connected but any path
connecting $(x,0)$ to $(y,T)$ must go "back in time", violating
Assumption~\ref{ass:m-good}. Certainly this implies there is no
dominating function $\tilde{m}$ for $m$ satisfying
Assumption~\ref{ass:m-good}.
\begin{example}
  \label{eg:1}
  Let $0<t_0<t_1<\dots<t_5<T$, $x_0<x_1<\dots<x_5$ and let
  $\Omega=(x_0,x_5)$. Consider the problem with
  $\mathcal{A}(t)u=-\partial u/\partial x^2$ in $\Omega\times(0,T)$
  and $\mathcal{B}(t)u=u$ in $\partial\Omega\times(0,T)$. Let
  $m=\one_X$ where $X$ is given by Figure \ref{fig:counter}.
  \begin{figure}[ht]
    \centering
    \begin{tikzpicture}
      \pgfmathsetmacro\T{4}%
      \pgfmathsetmacro\w{1.6}%
      \foreach \k in{0,1,...,5}{%
        \draw[dotted] (\w,{\T*(2*\k+1)/12}) -- (-\w,{\T*(2*\k+1)/12})
        node[left] {$t_{\k}$};%
        \path ({-\w+(2*\k*\w/5)},0) node[below] {$x_{\k}$};%
      }%
      \path (-\w,0) node[left=3pt] {$0$};%
      \path (-\w,\T) node[left=1pt] {$T$};%
      \draw[fill=gray] (-\w,9*\T/12) -| (\w/5,5*\T/12) --
      (3*\w/5,5*\T/12) |- (-\w,11*\T/12);%
      \draw[fill=gray] (\w,3*\T/12) -| (-\w/5,7*\T/12) --
      (-3*\w/5,7*\T/12) |- (\w,\T/12);%
      \draw (-\w,0) rectangle (\w,\T);%
    \end{tikzpicture}
    \caption{Graph of set $X$ (shaded) and set $D_m$ (white).}
    \label{fig:counter}
  \end{figure}
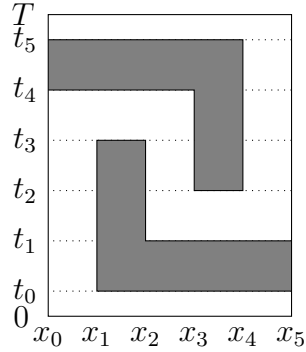
  If $\Delta_\Omega$ is the operator associated with
  $(\mathcal{A},\mathcal{B})$ on $\Omega$, then clearly
  \begin{displaymath}
    U_\lambda(t,s)=
    \begin{cases}
      \exp\bigl((t-s)(\Delta_\Omega-\lambda\one_{[x_1,x_5)}\bigr)
      &\text{if $t_0<s\le t<t_1$,}\\
      \exp\bigl((t-s)(\Delta_\Omega-\lambda\one_{[x_1,x_2]}\bigr)
      &\text{if $t_1<s\le t<t_2$,}\\
      \exp\bigl((t-s)(\Delta_\Omega-\lambda\one_{[x_1,x_2]\cup[x_3,x_4]}\bigr)
      &\text{if $t_2<s\le t<t_3$,}\\
      \exp\bigl((t-s)(\Delta_\Omega-\lambda\one_{[x_3,x_4]}\bigr)
      &\text{if $t_3<s\le t<t_4$,}\\
      \exp\bigl((t-s)(\Delta_\Omega-\lambda\one_{(x_0,x_4]}\bigr) &\text{if
        $t_4<s\le t<t_5$.}
    \end{cases}
  \end{displaymath}
  Using strong continuity we can extend extends these to $t_j\le s\le
  t\le t_{j+1}$.  Standard absorption semigroup techniques (see
  \cite{arendt:93:asd}) show that, if $\Omega_0\subset\Omega$ is open,
  \begin{equation}
    \label{eq:convergence}
    \exp\bigl(t(\Delta_\Omega-\lambda\one_{\Omega_0^c})\bigr)
    \to e^{t\Delta_{\Omega_0}}\one_{\Omega_0}
  \end{equation}
  in $\mathcal{L}(L^p(\Omega))$ for $1<p<\infty$, where
  $\Delta_{\Omega_0}$ is the Dirichlet Laplacian on
  $\Omega_0$. Moreover, if $A,B\subseteq\Omega$ are open and $A\cap
  B=\emptyset$, then
  \begin{displaymath}
    e^{t\Delta_{A\cup B}}
    =e^{t\Delta_A}\one_A+e^{t\Delta_B}\one_B.
  \end{displaymath}
  Taking $u_0\in L^p(\Omega)$ and defining
  $u_j:=U_\infty(t_j,t_{j-1})u_{j-1}$, we then find that $u_1$, $u_2$
  and $u_3$ vanish outside $(x_0,x_1)$, implying $u_4$ vanishes outside
  $(x_0,x_3)$. But then
  \begin{displaymath}
    u_5=U_\infty(t_5,t_4)u_4
    =e^{(t_5-t_4)\Delta_{(x_4,x_5)}}\one_{(x_4x_5)}u_4=0
  \end{displaymath}
  and hence $U_\infty(t_5,t_0)=0$. In particular, this means
  $U_\infty(T,0)=0$ and so the limit problem is trivial.
\end{example}

\pdfbookmark[1]{\refname}{biblio}
\providecommand{\mathbb}[1]{\mathbf{#1}}\providecommand{\cprime}{$'$}
\providecommand{\bysame}{\leavevmode\hbox to3em{\hrulefill}\thinspace}

\end{document}